\documentclass{amsart}
\usepackage{amsmath,amssymb,amsthm}

\newtheorem{theo}{Theorem}
\newtheorem{defi}[theo]{Definition}
\newtheorem{lemm}[theo]{Lemma}
\newtheorem{prop}[theo]{Proposition}

\newtheorem{exam}[theo]{Example}
\begin{document}
\title{Locally Compact Separable Abelian Group Actions on Factors with the Rohlin Property}
\author{Koichi Shimada}\thanks{Communicated by N. Ozawa. Received August 24. 2013. Revised December 3. 2013 \\ Department of Mathematical Sciences, the University of Tokyo, Komaba, Tokyo, 153-8914, Japan \\ shimada@ms.u-tokyo.ac.jp}

\date{}

\begin{abstract}
We show  a classification theorem for actions with the Rohlin property of locally compact separable abelian groups on factors. This is a generalization of the recent work due to Masuda--Tomatsu on Rohlin flows.
\end{abstract}

\maketitle
\section{Introduction}
Studying group actions is one of the most interesting topics in theory of operator algebra.  Since Connes \cite{C}, \cite{C2} has completely classified single automorphisms of the approximately finite dimensional (hereafter abbreviated by AFD) factor of type $\mathrm{II}_1$ up to cocycle conjugacy, the classification of group actions has been remarkably developed; discrete amenable group actions on AFD factors are completely classified by many hands \cite{KtST}, \cite{KwhST}, \cite{M}, \cite{O},  \cite{ST}, \cite{ST2}, and there has been great progress in the classification of compact group actions on AFD factors by Jones--Takesaki \cite{JT}, Kawahigashi--Takesaki \cite{KwhT}, Masuda--Tomatsu \cite{MT1}, \cite{MT2}, \cite{MT3}. Then our next interest is the classification of actions of non-compact continuous groups. In the study of locally compact abelian group actions, many problems are left. As a step to understanding their actions, outer actions are now intensively studied. As a candidate for an appropriate outerness for flows, the Rohlin property was introduced by Kishimoto  \cite{Ksm}, Kawamuro \cite{Kwm} and Masuda--Tomatsu \cite{MT}, and the last two authors have succeeded in classifying Rohlin flows on von Neumann algebras in \cite{MT}. As mentioned in Problem 8.1 of \cite{MT}, the Rohlin property can be generalized for locally compact abelian group actions. Hence it is natural to extend the result of Masuda--Tomatsu for locally compact abelian group actions. In this direction, Asano \cite{A} has classified actions of $\mathbf{R}^n$ with the Rohlin property. In this paper, we will generalize results in \cite{A} and \cite{MT}, and classify actions with the Rohlin property of locally compact separable abelian groups on factors (Theorem \ref{main}). This gives an answer to Problem 8.1 of \cite{MT}.

This paper is organized as follows. In Section \ref{The Rohlin Property}, we recall the definition of ultraproduct von Neumann algebras and the Rohlin property. In Section \ref{classification}, we will present a proof of the main theorem. The proof is basically modeled after that in \cite{MT}. Hence we will explain what are different from their proof. For example, contrary to $\mathbf{R}$, some locally compact abelian groups do not have enough compact quotients. We need to deal with this kind of problems. In Section \ref{ex}, we will see some examples of Rohlin actions.

\section{Preliminary}
\label{The Rohlin Property}

\subsection{Notations}
Let $M$ be a von Neumann algebra. We denote the set of unitaries of $M$ by $U(M)$. For $\phi \in M_*$ and $a\in M$, set $[\phi, a]:=a\phi -\phi a$. For $\phi \in M_*^+$ and $x\in M$, set 
\[ \| x\|_\phi ^\sharp :=\sqrt{\frac{\phi (x^*x+xx^*)}{2}}. \]

This $\|\cdot \|_\phi^\sharp $ is a seminorm on $M$. If $\phi $ is faithful, then this norm metrizes the strong* topology of the unit ball of $M$.

\subsection{Ultraproduct von Neumann Algebras}
 First of all, we recall ultapruduct von Neumann algebras. Basic references are Ando--Haagerup \cite{AH} and Ocneanu \cite{O}. Let $\omega $ be a free ultrafilter on $\textbf{N} $ and $M$ be a separable von Neumann algebra. We denote by $l^{\infty }(M)$ the $\mathrm{C}^*$-algebra consisting of all norm bounded sequences in $M$. Set
\[ I_{\omega }:=\{ (x_n) \in l^{\infty }(M) \mid \text{strong*-lim}_{n\to \omega }x_n=0 \}, \]
\begin{align*}
 N_{\omega }:=\{ (x_n) \in l^{\infty }(M) \mid \ &\text{for all } (y_n)\in I_{\omega }, \\
                                            &\text{ we have } (x_ny_n)\in I_\omega \text{ and }(y_nx_n)\in I_\omega \} ,
 \end{align*}
\[ C_{\omega }:=\{ (x_n) \in l^{\infty }(M) \mid \text{for all } \phi \in M_{*}, \text{we have} \lim _{n\to \omega }\|[\phi ,x_n]\| =0 \}. \]
Then we have $I_{\omega }\subset C_{\omega }\subset N_{\omega }$ and $I_{\omega }$ is a closed ideal of $N_{\omega}$. Hence we can define the quotient $\mathrm{C}^{*}$-algebra $M^{\omega }:=N_{\omega }/I_{\omega }$. Denote the canonical quotient map $N_{\omega }\to M^{\omega }$ by $\pi $. Set $M_{\omega }:=\pi (C_{\omega })$. Then $M_{\omega }$ and $M^{\omega }$ are von Neumann algebras as in Proposition 5.1 of Ocneanu \cite{O}. 

Let $\tau ^\omega \colon M^\omega \to M$ be the map defined by $\tau ^\omega (\pi((x_n)))=\lim _{n\to \omega }x_n$. Here, the limit is taken in the weak topology of $M$. This map is a faithful normal conditional expectation (see Subsection 2.4 of \cite{MT}). 

Let $\alpha $ be an automorphism of $M$. We define an automorphism $\alpha ^{\omega }$ of $M^{\omega }$ by $\alpha^{\omega }(\pi ((x_n)))=\pi ((\alpha (x_n)))$ for $\pi ((x_n))\in M^{\omega}$. Then we have $\alpha^{\omega }(M_{\omega })=M_{\omega }$. By restricting $\alpha ^{\omega }$ to $M_{\omega }$, we define an automorphism $\alpha _{\omega }$ of $M_{\omega }$. Hereafter we omit $\pi$ and denote $\alpha ^{\omega }$ and $\alpha _{\omega }$ by $\alpha $ if no confusion arises. 

\subsection{The Rohlin Property}

 Next, we recall the Rohlin property. A basic reference is \cite{MT}. In the previous subsection, we have seen that it is possible to lift automorphisms of von Neumann algebras on their ultraproducts. Hence it is natural to consider lifts of actions of locally compact abelian groups on $M^\omega $ and $M_\omega $. However, lifts may not be continuous. Instead of considering $\alpha ^\omega $ on whole $M^\omega$, we consider their continuous part. Let $G$ be a locally compact separable abelian group. In the rest of the paper, we always assume that groups and von Neumann algebras are separable, except for ultaproduct von Neumann algebras. We denote the group operation of $G$ by $+$. Let $d$ be a translation invariant metric on $G$ (see Theorem 8.3 of \cite{HR}). Choose a normal faithful state $\varphi $ on $M$. For an action $\alpha $ of $G$ on a von Neumann algebra $M$, set 

\begin{multline*}
 M^\omega _{\alpha }:=\{ (x_n)\in M^\omega \mid \text{for \ each } \epsilon >0,\text{ there \ exists } \delta >0 \text{ such that } \\
      \{n\in \textbf{N}\mid \| \alpha _t(x_n)-x_n\|_{\varphi }^{\sharp } <\epsilon   \ \mathrm{ for }\ t\in G\ \mathrm{with} \ d(0,t)<\delta \}\in \omega \}, 
\end{multline*}

\begin{multline*}
 M_{\omega ,\alpha }:=\{ (x_n)\in M_\omega \mid \text{for \ each } \epsilon >0,\text{ there exists } \delta >0 \text{ such that } \\
      \{n\in \textbf{N}\mid \| \alpha _t(x_n)-x_n\|_{\varphi }^{\sharp } <\epsilon   \ \mathrm{ for }\ t\in G\ \mathrm{with} \ d(0,t)<\delta \}\in \omega \}. 
\end{multline*}
Since all metrics on $G$ are mutually equivalent, this definition does not depend on the choice of $d$. The condition appearing in the definition of $M^\omega _\alpha $ means the $\omega$-equicontinuity of the family of maps $\{G\ni t \mapsto \alpha _t(x_n)\}$ (see Definition 3.1 and Lemma 3.2 of \cite{MT}). Now, we will define the Rohlin property.

\begin{defi}
An action $\theta $ of a locally compact abelian group $G$ on a von Neumann algebra $M$ is said to have the Rohlin property if for each $p\in \hat{G}$, there exists a unitary $u$ of $M_{\omega, \theta }$ satisfying $\theta _t(u)=\langle t,-p\rangle u$ for all $t\in G$.
\end{defi}
The Rohlin property is also defined for Borel cocycle actions (see Definition 3.4 and Definition 4.1 of \cite{MT}). For actions, by the same argument as in the proof of Proposition 3.5 of \cite{MT}, it is shown that the two definitions coincide.

\section{A Classification Theorem}
\label{classification}
Let $G$ be a locally compact abelian group. Let $\alpha ^1$ and $\alpha ^2$ be two actions of $G$ on a von Neumann algebra $M$. Two actions $\alpha ^1$ and $\alpha ^2$ are said to be  \textit{cocycle conjugate} if there exist an $\alpha ^2$-cocycle $u$ and an automorphism $\sigma $ of $M$ satisfying $\mathrm{Ad}u_t \circ \alpha ^2_t =\sigma \circ \alpha ^1_t \circ \sigma ^{-1}$ for all $t\in G$. If $\sigma $ can be chosen to be  approximately inner, then $\alpha ^1$ is said to be \textit{strongly cocycle conjugate} to $\alpha ^2$ (see Subsection 2.1 of \cite{MT}).

Our main theorem of this paper is the following.

\begin{theo}
\label{main}
Let $G$ be a locally compact abelian group. Let $\alpha $ and $\beta $ be actions of $G$ with the Rohlin property on a factor $M$. Then $\alpha $ and $\beta$ are strongly cocycle conjugate if and only if $\alpha _t \circ \beta _{-t} \in \overline {\mathrm{Int}}(M)$ for all $t\in G$.
\end{theo}

In the rest of this section, we will present a proof of  this theorem. The proof is modeled after that in \cite{MT}. However, at some points of the proof, we need to deal with problems different from those in their proof. One of the problems is that some locally compact abelian groups do not have enough compact quotients. Instead, we consider compact quotients of compactly generated clopen subgroups. By Theorem 9.14 of Hewitt--Ross \cite{HR}, a compactly generated subgroup is isomorphic to $\mathbf{R}^n \times K \times \mathbf{Z}^m$ for some compact abelian group $K$ and non-negative integers $n$, $m$. We deal with this problem in Subsection \ref{cv}.

\subsection{Lifts of Borel Unitary Paths}

The first step of our proof of Theorem \ref{main} is to find a representing unitary sequence $\{ u^\nu _t\}$ for a Borel map $U_t\colon G\to U(M^\omega _\theta)$ so that the family $\{ t\mapsto u_t^\nu \}$ is ``almost'' $\omega$-equicontinuous. More precisely, we have the following.

\begin{lemm}
\textup{(See Lemma 3.24 of \cite{MT})} Let $(\theta , c)$ be a Borel cocycle action of a locally compact abelian group $G$ on a factor $M$. Suppose that $U\colon G\to M_\theta ^\omega $ is a Borel unitary map. Let $H$ be a compactly generated clopen subgroup of $G$, which is isomorphic to $\mathbf{R}^n\times K \times \mathbf{Z}^m$ for some non-negative integers $n$, $m$ and a compact abelian group $K$. Let $L$ be a subset of $H$ of the form 
\[ L=[0, S_1)\times \cdots \times [0, S_n) \times K \times [0, N_1)\times \cdots [0, N_m)\]
 when we identify $H$ with $\mathbf{R}^n \times K\times \mathbf{Z}^m$. Then for any $\delta >0$ with $0<\delta <1$ and a finite set $\Phi $ of $M_*^+$, there exist a compact subset $I$ of $L\times L$, a compact subset $C$ of $L$ and a lift $\{u_t^\nu \}$ of $U$ satisfying the following conditions.

\textup{(1)} We have $\pi _\omega \bigl((u_t^\nu )_\nu \bigr)=U_t$ for almost every $t\in L$ and the equality holds for all $t\in C$.

\textup{(2)} We have $\mu _G(L\setminus C)<\delta $, where $\mu _G$ is the Haar measure on $G$.

\textup{(3)} For all $\nu \in \mathbf{N}$, the map $L \ni t \mapsto u_t^\nu $ is Borel and its restriction to $C$ is strongly continuous.

\textup{(4)} The family of maps $\{C\ni t \mapsto u_t^\nu \} _\nu $ is $\omega$-equicontinuous.

\textup{(5)} We have $(\mu _G\times \mu _G) (I) \geq (1-\delta )(\mu _G\times \mu _G)(L\times L)$.

\textup{(6)} The family of maps $\{ I\ni (t,s)\mapsto u_t^\nu \theta _t(u_s^\nu )c(t,s)(u_{t+s}^\nu)^* \} _\nu $ is $\omega $-equicontinuous.

\textup{(7)} The following limit is the uniform convergence on $I$ for all $\phi \in \Phi$.
\[ \lim _{\nu \to \omega} \| u_t^\nu \theta _t (u_s^\nu )c(t,s)(u_{t+s}^\nu )^*-1 \| _\phi ^\sharp =\| U_t \theta _t (U_s )c(t,s)U_{t+s}^* -1\| _{\phi ^\omega }^\sharp . \]
\label{10}
\end{lemm}
The proof is similar to that of Lemma 3.24 of \cite{MT}. Here, we only prove the following lemma, which corresponds to Lemma 3.21 of \cite{MT}. The proof is a simple approximation by Borel simple step functions.

\begin{lemm} \textup{(See also Lemma 3.21 of \cite{MT})}
Let $G$ be a locally compact abelian group, $\theta \colon G\to \mathrm{Aut}(M)$ be a Borel map and $U\colon G\to M_{\theta}^\omega $ be a Borel unitary map. Then for any Borel subset $L$ of $G$ with $0<\mu _G(L)<\infty $ and for any $\epsilon >0$, there exist a compact subset $C$ of $L$ and a sequence $\{u_t^\nu \} _{\nu \in \textbf{N}}$ of unitaries of $M$ for any $t\in L$ which satisfy the following conditions. 

\textup{(1)} We have $\pi _\omega \bigl((u_t^\nu )_\nu \bigr)=U_t$ for almost every $t\in L$ and the equality holds for all $t\in C$.

\textup{(2)} We have $\mu _G(L\setminus C)<\epsilon $.

\textup{(3)} For all $\nu \in \mathbf{N}$, the map $L \ni t \mapsto u_t^\nu $ is Borel and its restriction to $C$ is strongly continuous.

\textup{(4)} The family of maps $\{C\ni t \mapsto u_t^\nu \} _\nu $ is $\omega$-equicontinuous.
\label{7}
\end{lemm}
\begin{proof}
By the same argument as in the proof of Lemma 3.21 of \cite{MT}, it is shown that there exists a sequence $\{ L_n\}$ of compact subsets of $L$ satisfying the following conditions.

(1) We have $L_i\cap L_j=\emptyset $ for $i\not =j$.

(2) We have $\mu _G (L\setminus \bigcup _{j=1}^\infty L_j)=0$.

(3) The map $U|_{L_i}$ is continuous for each $i$.

Hence we may assume that $L$ is compact and that $U\mid _L$ is strongly continuous. Let $\psi\in M_*$ be a normal faithful state. For each $t\in L$, take a representing unitary $\{\tilde{U}_t^\nu\} _\nu $ of $U_t$. Note that $t\mapsto \tilde{U}_t^\nu$ may not be Borel measurable.  We first show the following claim.

\textbf{Claim}. For each $k\in \mathbf{N}$, there exist $N_k \in \mathbf{N}$, $F_k \in \omega $, a finite subset $A_k $ of $L$, a finite Borel partition $P^k:=\{K_l^k \}_{l=1}^{n_k}$ of $L$ and a compact subset $C_k$ of $L$ satisfying the following conditions. 

(1) For $s,t \in L$ with $d(s,t)\leq 1/N_k$, we have $\| U_s-U_t\| _{\psi^\omega }^\sharp <1/2k$. 

(2) We have $N_k >N_{k-1}, 2/N_k +1/(2N_{k-1})<1/N_{k-1}$ for all $k$.

(3) We have $[k, \infty ) \supset F_{k-1} \supsetneq F_k$ for all $k$.

(4) We have $A_k \supset A_{k-1}$ for all $k$.

(5) We have $\bigcup _{j=1}^{\infty} A_j \subset C_k$, $C_{k+1}\subset C_k$, $\mu _G(L\setminus C_k)<\epsilon (1-2^{-k})$ for all $k$ and $C_k \cap K^k_l$'s are also compact for all $k\in \mathbf{N}$, $l=1,\cdots n_k$.

(6) For each $k$, the partition $P^{k+1}$ is finer than $P^k$ and for each $k\in \mathbf{N}$, $l=1,\cdots , n_k$, we have $A_k \cap K_l^k =\{ t_{k,l}\} (=\{\mathrm{pt}\}) $.

(7) For $s,t \in K_l^k$, we have $d(s,t)\leq 1/N_k$.

(8) For $s,t \in A_k$, $\nu \in F_k$, we have $\| \tilde{U}_s^\nu-\tilde{U}_t^\nu \| ^{\sharp}_{\psi }< \| U_s-U_t\| _{\psi ^\omega }^\sharp + 1/(2k)$.

\textit{Proof of Claim}. First of all, choose a sequence $\{ N_k \}_{k=1}^\infty \subset \mathbf{N}$ so that the sequence satisfies conditions (1) and (2). Next, we take $P^k$'s. Assume that $P^1, \cdots P^k$ are chosen so that they satisfy condition (7) and that $P^{j+1}$ is a refinement of $P^j$ for $j=1, \cdots , k-1$. By compactness of $L$, there exists a family of finite balls $\{ B_f\} _{f\in F}$ of radius $1/(2N_{k+1})$ of $L$ which covers $L$. This $\{ B_f\} _{f\in F}$ defines a partition $\{ \tilde{B}_{f'}\}_{f'\in F'} $ of $L$. Then $P^{k+1}:=\{K_k^l \cap \tilde{B}_{f'}\} _{f'\in F', l=1, \cdots , n_k}$ is a refinement of $P^k$, which satisfies condition (7). Next, we take $C_k$'s. Set $C_0:=L$ and $C^0_1:=C_0$.  By Lusin's theorem, for each $l=1, \cdots , n_k$, $k\in \mathbf{N}$, there exists a compact subset $C_l^k$ of $K_l^k$ which satisfies the following conditions.

(1) We have $C^{k+1}_l \subset C^k _{l'}$ if $K^{k+1}_l \subset K^k_{l'}$.

(2) We have $\mu _G((K^{k+1}_l\cap C^k_{l'})\setminus C^{k+1}_l)<2^{-(k+1)}\epsilon /n_{k+1}$ if $K^{k+1}_l \subset K^k_{l'}$.

Set $C_k:=\bigcup _{l=1}^{n_k}C^k_l$ for each $k\in \mathbf{N}$. Since $C^k_l$'s are compact, $C_k$ is also compact. On the other hand, we have 
\begin{align*}
\mu _G(C_j \setminus C_{j+1}) &=\sum _{l=1}^{n_{j+1}}\mu _G((K_l^{j+1}\cap C_j)\setminus C_{j+1}) \\
                            &=\sum _{l=1}^{n_{j+1}}\mu _G((K_l^{j+1}\cap C^j_{l'})\setminus C_{j+1})\\
                            &=\sum _{l=1}^{n_{j+1}}\mu _G((K_l^{j+1}\cap C^j_{l'})\setminus C_{l}^{j+1}) \\
                            &< n_{j+1}\frac{1}{n_{j+1}}2^{-(j+1)}\epsilon \\
                            &=2^{-(j+1)}\epsilon .
\end{align*}
In the above inequality, for each $l=1, \cdots n_{j+1}$, $l'\in \{1, \cdots , n_j\}$ is the unique number with $C^{j+1}_l \subset C^j_{l'}$. Hence we have 
\begin{align*}
\mu _G(L\setminus C_k) &\leq \sum _{j=0}^{k-1}\mu _G(C_j\setminus C_{j+1}) \\
                     &< \epsilon \sum _{j=0}^{k-1} 2^{-(j+1)} \\
                     &=\epsilon (1-2^{-k}).
\end{align*}
These $C_k$'s satisfy $C_{k+1}\subset C_k$ and $\mu _G(L\setminus C_k)<\epsilon (1-2^{-k})$, and we also have $C_k \cap K^k_l$($=C_l^k$) 's are compact. Next, we take $A_k$'s. For each $C^1_{l_1}\supset C^2_{l_2} \supset \cdots$, there exists $t_{l_1l_2\cdots}\in \bigcap _{k=1}^\infty C^k_{l_k}$ by compactness of $C^k_l$'s. By induction on $k$, it is possible to choose $A_k=\{t_{k,l}\} _{l=1}^{n_k}$ so that $A_k \subset A_{k+1}$ and that $t_{k,l} =t_{l_1l_2\cdots ll_{k+1}\cdots }$, i.e., $l_k=l$. These $A_k$'s satisfy conditions (4), (5) and (6). We may choose $F_k$'s so that they satisfy conditions (3) and (8). This completes the proof of Claim. \qed

Now,  we return to the proof of Lemma \ref{7}. For $t\in L$, set $U_t^{k, \nu }:=\tilde{U}_{t_{k,l}}^\nu$ if $t\in K_l^k$, $u^\nu _t:=U_t^{k, \nu }$ for $\nu \in F_k \setminus F_{k+1}$. Set $C:=\bigcap _k C_k$. Then we have $\mu _G(L \setminus C)<\epsilon $ by condition (5) of Claim. Since $U_t^{k,\nu}$'s are continuous on each $K^k_l\cap C_k$($=C^k_l$) and $C^k_1, \cdots C^k_{n_k}$ are compact, $U^{k, \nu}_t$'s are continuous on each $C_k$. Hence they are continuous on $C$. Hence by the same argument as in the proof of Lemma 3.21 of \cite{MT}, the map $C\ni t\mapsto u_t^\nu $ is strongly continuous for each $\nu \in \mathbf{N}$. Then by the same argument as in Lemma 3.21 of \cite{MT}, it is possible to see that $\{ C\ni t\mapsto u_t^\nu \}_\nu$ is $\omega$-equicontinuous and that $\pi _\omega (u_t^\nu )=U_t$ for all $t\in C$. Now, we have chosen $\{u_t^\nu\}_\nu$ and $C$ so that they satisfy conditions (2),(3) and (4) of Lemma \ref{7} and the following condition.

(1)' We have $\pi _\omega ((u_t^\nu)_\nu )=U_t$ for $t\in C$.

Hence what remains to be done is to replace $\{u_t^\nu \}_\nu$ so that $\pi _\omega ((u_t^\nu )_\nu )=U_t$ for almost all $t\in L$. By repeating the same process, we can find a sequence of compact subsets $\{ D_n \}_{n=0}^{\infty}$ of $L$ and a sequence of strongly continuous maps $\{ D_n \ni t\mapsto u_t^{n, \nu}\in U(M)\} _{n,\nu=0}^\infty$ which satisfy the following conditions.

(1) We have $\mu _G(L\setminus (\bigcup _{n=0}^\infty D_n ))=0$ and $D_n $'s are mutually disjoint.

(2) We have $\pi _\omega ((u_t^{n, \nu})_\nu )=U_t$ for $t\in D_n$.

(3) We have $D_0=C$ and $u_t^{0,\nu }=u_t^\nu |_C$ for all $\nu \in \mathbf{N}$.

Set $u_t^\nu :=u_t^{n, \nu }$ for $t\in D_n$. This $\{ u_t^\nu \}_\nu$ satisfies all conditions of Lemma \ref{7}. 
\end{proof}

\subsection{The Averaging Technique}
\label{ar}
Next, we show the ``averaging technique''. For the $\mathbf{R}$-action case, this means that it is possible to embed $(M\otimes L^\infty ([0,S)), \theta \otimes \mathrm{translation})$ into $(M^{\omega}_{\theta}, \theta)$ for any $S>0$. This is a key lemma for the classification theorem. For the general case, the following lemma corresponds to this.

\begin{lemm}
\label{averaging}
Let $G$ be a locally compact abelian group and $\theta$ be an action with the Rohlin property of $G$ on a factor $M$. Let $L$ be a subset of $G$ with the following properties. 

\textup{(1)} There exists a compactly generated clopen subgroup $H$ of $G$, which is isomorphic to $\mathbf{R}^n \times K \times \mathbf{Z}^m$ for some compact group $K$ and non-negative integers $n$, $m$. 

\textup{(2)} The set $L$ is a subset of $H$. When we identify $H$ with $\mathbf{R}^n \times K \times \mathbf{Z}^m$, $L$ is of the form $[0,S_1) \times \cdots \times [0,S_n)\times K \times [0, N_1)\times \cdots \times [0,N_2)$. Note that $L$ can be thought of as a quotient group of $H$. 

Then there exist a unitary representation $\{ u_k\} _{k\in \hat{L}}$ of $\hat{L}$ on $M_{\omega , \theta}$ and an injective *-homomorphism $\Theta :M\otimes L^\infty(L) \to M^{\omega}_{\theta}$ with the following properties.

\textup{(1)} We have $\theta _t \circ \Theta =\Theta \circ (\theta _t \otimes \gamma _t)$. Here, $\gamma :H\curvearrowright L^\infty (L)$ denotes the translation.

\textup{(2)} We have $\Theta (a\otimes \langle \cdot ,k\rangle)=au_k$ for $a\in M$, $k\in \hat{L}$.

\textup{(3)} We have $\tau ^\omega \circ \Theta =\mathrm{id}_M\otimes \mu _L$,  where $\mu _L$ denotes the normalized Haar measure on $L$, which is the normalization of the restriction of a Haar measure on $G$, and $\tau ^\omega $ is the normal faithful conditional expectation as in Section \ref{The Rohlin Property}.
\end{lemm}

In order to show this, by the same argument as in Lemma 5.2 of \cite{MT} (in this part, we use the fact that $M$ is a factor), it is enough to show the following proposition.

\begin{prop}
 Let $\theta :G \curvearrowright M$ be an action with the Rohlin property of a locally compact abelian group $G$ on a factor $M$ and $L\subset H$ be subsets of $G$ as in the above lemma. Then there exists a family of unitaries $\{u_k\}_{k\in \hat{L}}\subset U(M_{\omega , \theta})$ with  the following properties.

\textup{(1)} We have $\theta _t(u_k)=\langle t, k\rangle u_k$ for $t\in H$.  

\textup{(2)} The map $k \mapsto u_{k}$ is an injective group homomorphism. 
 \label{1}
\end{prop}
To show the above proposition, we need to prepare some lemmas. In the rest of this subsection, $\theta$, $G$, $H$ and $L$ are as in Proposition \ref{1}.

\begin{lemm}
\label{1dim}
Let $C$ be a subgroup of $\hat{L}$ isomorphic to $\mathbf{Z}/l\mathbf{Z}$. Then there exists a family of unitaries $\{u_k\}_{k\in C}\subset M_{\omega, \theta }$ with the following properties.

\textup{(1)} We have $\theta _t(u_k)=\langle t, k\rangle u_k$ for $t\in H$. 

\textup{(2)} The map $C\ni k \mapsto u_{k}$ is an injective group homomorphism. 
\end{lemm}
\begin{proof}
Let $p$ be a generator of $C$. Since $\theta $ has the Rohlin property, there exists a unitary $w$ of $M_{\omega, \theta}$ satisfying $\theta_t(w)=\langle t, p \rangle w$ for $t\in H$. Since $w^l \in M_{\omega, \theta} ^\theta$, there exists a unitary $v$ of $M_{\omega, \theta} ^\theta \cap \{w\}'$ such that $v^{-l}=w^l$. Set $u:=vw$ and $u_k:=u^k$. Then the family $\{ u_k \}_{k\in \mathbf{Z}/l\mathbf{Z}}$ does the job.   
\end{proof}

By the same argument as in the proof of Lemma 3.16 of \cite{MT}, we have the following lemma. See also Lemma 5.3 of Ocneanu \cite{O}, Lemma 3.16 of \cite{MT}.

\begin{lemm}\textup{(Fast reindexation trick.)}
Let $\theta $ be an action of $G$ on a von Neumann algebra $M$ and let $F\subset M^\omega$ and $N\subset M_\theta^\omega$ be separable von Neumann subalgebras. Suppose that the subalgebra $N$ is globally invariant by $\theta $. Then there exists a faithful normal *-homomorphism $\Phi :N\to M_\theta^\omega $ with the following properties.
\begin{gather*}  
\Phi =\mathrm{id} \ \text{on $F\cap M$,} \\
\Phi (N\cap M_{\omega, \theta })\subset F'\cap M_{\omega, \theta}, \\
\tau ^\omega (\Phi (a)x)=\tau ^\omega (a)\tau ^\omega (x) \ \text{for all $a\in N$, $x\in F$,} \\
\theta_t \circ \Phi =\Phi \circ \theta _t \ \text{on $N$ for all $t \in L$.} 
\end{gather*}

\end{lemm}
\begin{lemm}
\label{finitely gen}
Let $C$ be a subgroup of $\hat{L}$ of the form $\mathbf{Z}^n\times F$, where $F:=\bigoplus _{k=1}^m \mathbf{Z}/(l_k\mathbf{Z})$ is a finite abelian group. Then there exists a family of unitaries $\{u_k\}_{k\in C} \subset M_{\omega , \theta }$ which satisfies the following conditions. 

\textup{(1)} We have $\theta _t(u_k)=\langle t, k\rangle u_k$ for $t\in H$. 

\textup{(2)} The map $k \mapsto u_{k}$ is an injective group homomorphism.
\end{lemm}
\begin{proof}
Let $\{p_1,\cdots ,p_n,q_1,\cdots ,q_m\}$ be a base of $\hat{C}$. Then there exist unitaries $\{u_i\}_{i=1}^{n}$ and $\{v_j\}_{j=1}^m$ with $\theta _t(u_i)=\langle t,p_i \rangle u_i$, $\theta _t(v_j)=\langle t, q_j \rangle v_j$ for $t\in H$. By Lemma \ref{1dim}, we may assume that $ v_j^{l_j}=1$. By using the fast reindexation trick, it is possible to choose $\{u_i\}_{i=1}^{n}$ and $\{v_j\}_{j=1}^m$ so that they mutually commute.   
\end{proof}

Now, we prove Proposition \ref{1}.

\begin{proof}
Let $ \psi \in M_*$ be a normal faithful state and let $\Phi =\{ \phi _m \} $ be a countable dense subset of the unit ball of $M_*$. There exists an increasing sequence $\{ C_\nu \}$ of finitely generated subgroups of $\hat{L}$ satisfying $\hat{L}=\bigcup_{\nu=1}^\infty C_\nu$. Then by the structure theorem of finitely generated abelian groups and the above lemma, for each $\nu$, there exists a family of unitaries $\{ u_{k}^\nu \} _{k\in C_\nu }\subset U(M_{\omega , \theta})$ with $C_\nu \ni k \mapsto u_{k}^\nu $ satisfying conditions (1) and (2) of Lemma \ref{finitely gen}. For each $k\in \hat{L}$, set a sequence $\{k_\nu \}$ of $\hat{L}$ as follows.
\begin{equation*}
 k_\nu =\begin{cases}
           k & \text{if $ k \in C_\nu $} \\
           0 & \text{if $ k \not \in C_\nu $ .}
        \end{cases} 
\end{equation*} 
For each $\nu\in \mathbf{N}$, $k \in C_\nu $, take a representing sequence $\{u_{k}^{\nu, n} \}$ of $u_{k}^\nu $. Take a sequence $\{ E_\nu \}$ of finite subsets of $\hat{L}$ satisfying $\bigcup E_\nu =\hat{L}$, $E_\nu \subset C_\nu $ for all $\nu \in \mathbf{N}$. By Lemma 3.3 of \cite {MT}, the convergence 
\[ \lim _{ n \to \omega } \| \theta _t(u_{k }^{\nu ,n}) - \langle t, k \rangle u_{k }^{\nu ,n} \| ^{\sharp}_\psi =0 \]
is uniform for $t\in L$. Hence it is possible to choose $F_\nu \in \omega $ ($\nu =1,2,3, \cdots $) so that 
\begin{gather}
F_\nu \subsetneq F_{\nu -1} \subset [\nu -1, \infty ), \  \nu =2,3,\cdots , \\
\| u_k^{\nu ,n}u_l^{\nu ,n}-u_{k+l}^{\nu ,n} \| _{\psi}^\sharp <1/\nu , \ k,l \in E_\nu , \ n\in F_\nu ,\\
\| [ \phi _m, u_k^{\nu ,n}] \| <1/\nu , \ k\in E_\nu ,\ m\leq \nu ,\ n \in F_\nu ,\\
\| \theta _t(u_k^{\nu ,n}) -\langle t, k \rangle u_k^{\nu ,n} \| _\psi ^\sharp <1/\nu , \ k \in E_\nu , \ t\in L, \ n\in F_\nu .  
\end{gather}
Set $(u_k)_n:=u_{k_\nu}^{\nu, n}$ for $n\in F_\nu \setminus F_{\nu +1}$. We show that $u_k:=\{(u_k)_n\} $ is a desired family of unitaries. 

We will show $u_k \in M_{\omega }$. Fix $\mu \in \mathbf{N}$ and $k\in \hat{L}$. Then there exists $\nu \geq \mu $ with $k\in E_\nu$. Then for $n \in F_\nu$, there exists a unique $\lambda \geq \nu$ satisfying $n\in F_\lambda \setminus F_{\lambda +1}$. Then by the inequality (3), we have 
\[ \|[\phi _m, (u_k)_n ] \| =\| [\phi _m, (u_{k_\lambda }^{\lambda ,n} )] \| <1/\lambda \leq 1/\mu \]
for $m \leq \mu $. Thus we have $ u_k \in M_\omega$.

In a similar way to the above, we obtain $\theta _t(u_k)=\langle t,k \rangle u_k$, using the inequality (4). It is also possible to show that the map $\hat{L}\ni k \mapsto u_k$ is a unitary representation by using the inequality (2).
\end{proof}

\subsection{Cohomology Vanishing}
\label{cv}

By using Lemma \ref{averaging}, we show the following two propositions.  See also Theorems 5.5  and 5.11 of \cite{MT}, respectively. 

\begin{prop}
\label{2-cohomology}
\textup{(2-cohomology vanishing)} Let $(\theta , c )$ be a Borel cocycle action of a locally compact abelian group $G$ on a factor $M$. Suppose that $(\theta , c)$ has the Rohlin property. Then the $2$-cocycle $c$ is a coboundary, that is, there exists a Borel unitary map $v:G\to U(M)$ such that 
\[ v_t\theta _t (v_s)c(t,s )v_{t+s}^* =1 \]
for almost every $(t,s)\in G^2 $.
\end{prop}

Furthermore, if $\| c(t,s)-1\| ^\sharp _\phi $, $\|[c(t,s ), \phi ]\|$ ($\phi \in M_*$) are small, then it is possible to choose $v_t$ so that $\| v_t-1\| ^\sharp _\phi $ and $\|[v_t, \phi ] \|$ are small. We will explain this later.

\begin{prop}
\textup{(Approximate 1-cohomology vanishing)} \label{ap 1-coho} Let $\theta $ be an action with the Rohlin property of a locally compact  abelian group $G$ on a factor $M$. Let $\epsilon $, $\delta $ be positive numbers and $\Phi $ be a compact subset of the unit ball of $M_*$. Let $H$ be a compactly generated clopen subgroup of $G$, which is isomorphic to $\mathbf{R}^n \times K\times \mathbf{Z}^m$ for some compact  abelian group $K$ and non-negative integers $n$, $m$.  Let $T$, $L$  be subsets of $H$ which satisfy the following conditions.

\textup{(1)} When we identify $H$ with $ \mathbf{R}^n \times K\times \mathbf{Z}^m$, $L$ is of the form 
\[ [0,S_1)\times \cdots \times [0, S_n) \times K\times [0,N_1)\times \cdots \times [0,N_m),\]
 which implies that $L$ is a compact quotient of $H$.

\textup{(2)} We have
\[ \frac{\mu _G\bigl( \bigcap _{t\in T}(t+L)\bigr) }{\mu _G(L)} >1-4\epsilon ^2. \]
 Then for any $\theta$-cocycle $u_t$ with
\[ \frac{1}{\mu _G(L)} \int _L \|[u_t, \phi ]\| \ d\mu _G(t)<\delta \] 
for all $\phi \in \Phi$, there exists a unitary $w\in M$ such that 
\[ \| [w, \phi] \| <3\delta \ \mathrm{for \ all} \ \phi \in \Phi ,\]
\[ \| \phi  \cdot (u_t\theta _t(w)w^*-1)\| <\epsilon , \] 
\[ \| (u_t\theta _t(w)w^*-1) \cdot \phi \| <\epsilon \ \mathrm{for \ all} \ t\in T, \phi \in \Phi .\] 
\end{prop}
 By carefully examining arguments of the proofs of \cite{MT} Theorems 5.5 and 5.11, we notice that we need to choose  sequences $\{L_n\}$ and $\{ T_n\}$ of subsets of $G$ with the following properties.

(1) There exists an increasing sequence of compactly generated clopen subgroups $\{ H_k\}$ of $G$ with $\bigcup _k H_k=G$ and $L_k$, $T_k$ are subsets of $H_k$ and $T_k$'s are compact. When we identify $H_k$ with $\mathbf{R}^{n_k} \times K_k \times \mathbf{Z}^{m_k}$ for some compact abelian group $K_k$ and non-negative integers $n_k$, $m_k$,  the subset $L_k$ is of the form 
\[ [0,S_1)\times \cdots \times [0,S_{n_k})\times K_k\times [0,N_1)\times \cdots \times [0,N_{m_k}). \]
(2) The translation $H_k \curvearrowright L^\infty (L_k)$ is embedded into $(\theta, M_{\omega, \theta})$ (see Proposition \ref{1}).

(3) The quantity 
\[ \frac{\mu _G \bigl( L_k \setminus \bigcap _{t\in T_k+T_k}(t+L_k)\bigr)}{\mu _G(L_k)} \]
 is small.

(4) We have $L_k+T_k \subset T_{k+1}$.

(5) We have $T_k \subset T_{k+1}$ for all $k\in \mathbf{N}$ and $\bigcup _{k=1}^\infty T_k=G$.

For the $\mathbf{R}$-action case, $L_k=[0,s_k)$ and $T_k=[-t_k,t_k)$, $t_k \ll s_k \ll t_{k+1}$ do the job. In the following, we explain how to choose $L_k$'s and $T_k$'s for the general case. First, we show that there exists an increasing sequence $\{H_k\}$ of clopen subgroups of $G$ with the following conditions.

(6) For each $k$, the subgroup $H_k$ is compactly generated, which is isomorphic to $\mathbf{R}^{n}\times K_k \times \mathbf{Z}^{m_k}$ for some compact abelian group $K_k$. Note that the multiplicity $n$ of $\mathbf{R}$ of $H_k$ can be chosen to be independent on $k$ by Theorem 9.14 of \cite{HR}.

(7) We have $\bigcup _k H_k =G$.

This increasing sequence is chosen in the following way. There exists an increasing sequence $\{ O_k\}$ of open subsets of $G$ such that $\overline{O_k}$'s are compact, $0\in O_k$ for all $k\in \mathbf{N}$ and that $\bigcup _k \overline{O_k} =G$. For each $k\in \mathbf{N}$, let $H_k$ be the subgroup of $G$ generated by $\overline{O_k}$. We show that $H_k$ is clopen. If $t\in H_k$, then $t+O_k \subset H_k$. Hence this is open. Hence by Theorem 5.5 of \cite{HR}, $H_k$ is closed. By Theorem 9.14 of \cite{HR}, $H_k $ is of the form $\mathbf{R}^n \times K_k \times \mathbf{Z}^{m_k}$. 

Next, take two sequences $\{ L_k\}$ and $\{ T_k \}$ of subsets of $G$ and a decreasing sequence $\{ \epsilon _k \} \subset \mathbf{R} _{>0}$  with the following properties.

(8) The sets $L_k$, $T_k$ are subsets of $H_k$. When we identify $H_k$ with $\mathbf{R}^n \times K_k \times \mathbf{Z}^{m_k}$ for some compact abelian group $K_k$ and a non-negative integer $m_k$,  the subset $L_k$ is of the form 
\[ [0,S_1)\times \cdots [0,S_n)\times K_k\times [0,N_1)\times \cdots \times [0,N_{m_k}).\]
 Note that the way how to identify $H_k$ with $\mathbf{R}^n \times K_k \times \mathbf{Z}^{m_k}$ is not important. What is important is that $L_k$ is a quotient of a clopen subgroup of $G$.

(9) We have 
\[ \frac{\mu _G\bigl( L_k \setminus \bigcap _{t\in T_k+T_k}(t+L_k)\bigr)}{\mu _G(L_k)} >1-(\frac{\epsilon _k}{6\mu _G(T_k)^2})^2. \]

(10) We have $T_k+L_k \subset T_{k+1}$, $\bigcup _k T_k=G$ and $T_k$'s are compact.

(11) We have $0 < \epsilon _k <1/k$ and 

\[ \sum _{k=n+1}^\infty \sqrt{13\mu _G(T_k)\epsilon _k}<\epsilon _n. \]

From now on, we explain how to choose two sequences $\{L_k\} $ and $\{ T_k\}$. They are chosen in the following way. For each $k\in \mathbf{N}$, set $A_k :=\overline{O_k}$. Here, the set $O_k$ is chosen as in (7).

Assume that $(T_l, L_l, \epsilon _l)$, $l\leq k$ are chosen. Then since $A_{k+1}+T_k+\overline{L_k}$ is compact, it is possible to choose a subset $T_{k+1}\subset H_{k+1}$ so that when we identify $H_{k+1}$ with $\mathbf{R}^n\times K_{k+1} \times \mathbf{Z}^{m_{k+1}}$, $T_{k+1}$ is of the form 
\[ [-t_1,t_1]\times \cdots \times [-t_n,t_n]\times K_{k+1}\times [-M_1,M_1]\times \cdots \times [-M_{m_{k+1}},M_{m_{k+1}}]\]
 and that $A_{k+1}+T_k+L_k\subset T_{k+1}$. Since $\bigcup _k A_k =G$, we also have $\bigcup _k T_k=G$. Choose $\epsilon _{k+1}>0$ so that 
\[ \epsilon _{k+1}<\epsilon _k,\ \sqrt{13\mu _G(T_{k+1})\epsilon _{k+1}}<\epsilon _k/2^k. \]
 Choose $L_{k+1}\subset H_{k+1}$ so large that $L_{k+1}$ satisfies conditions (8) and (9). Thus we are done.   

By using the above sequences $\{ L_k\}$, $\{T_k\}$ instead of $\{S_k\}$ and $\{T_k\}$ of (5.14) of \cite{MT}, Propositions \ref{2-cohomology} and \ref{ap 1-coho} are shown by a similar argument to that of the proofs of Theorems 5.5 and 5.11 of \cite{MT}, respectively. Furthermore, it is possible to choose $v_t$ in Proposition \ref{2-cohomology} so that $v_t$ satisfies the following conditions.
 
(1) If for some $n\geq 2$ and a finite subset $\Phi \subset (M_*)_+$, we have 
\[ \int _{T_{n+1}} d\mu _G(t) \int _{T_{n+1}} d\mu _G(s) \| c(t,s)-1\| _\phi ^\sharp \leq \epsilon _{n+1} \]
for all $\phi \in \Phi$, then it is possible to choose $v_t$ so that
\[ \int _{T_n}\| v_t-1\| _\phi ^\sharp \ d\mu _G(t) <\epsilon _{n-1} d(\Phi ) ^{1/2}\]
for all $\phi \in \Phi$. Here, $d(\Phi)$ is defined in the following way. 
\[ d(\Phi ):=\mathrm{max}(\{1\} \cup \{\| \phi \| \mid \phi \in \Phi \}). \]

(2) If for some $n\geq 2$ and a finite subset $\Phi \subset M_*$, we have 
\[ \int _{T_{n+1}} d\mu _G(t) \int  _{L_{n+1}} d\mu _G(s) \|[c(t,s),\phi ]\| <\epsilon \]
for all $\phi \in \Phi$, then it is possible to choose $v_t$ satisfying
\[ \int _{T_n} \|[v_t, \phi ]\| \ d\mu _G(t) \leq (3\epsilon _{n-1} +3 \epsilon ) d(\Phi) \]
for all $\phi \in \Phi$.

\bigskip 

In the proof, the following points are slightly different. 

(1) The inequality corresponding to (5.12) of \cite{MT} is
\[ \frac{2\mu _G\bigl( L\setminus (\bigcap _{t\in T+T}t+L )\bigr) ^{1/2} }{\mu _G(L)^{1/2}}<\frac{\delta}{6\mu _G(T)^2}. \]

(2) We need to show a lemma which corresponds to Lemma 5.4 of \cite{MT}. In the proof, the inequality corresponding to (5.13) of \cite{MT} is the following.
\begin{align*}
& \| U_t\alpha _t(U_s)c(t,s)U_{s+t}^* -1\| _{\phi}^{\sharp} \\
&\leq \| \chi _{\bigcap _{t\in T+T}t+L}-1\| _{\phi \otimes \mu _L}^\sharp \\
&+\| \chi _{L\setminus (\bigcap _{t\in T+T}t+L)}\bigl( \mathrm{(a \ unitary \ valued \ function)}-1\bigr) \| _{\phi \otimes \mu _L}^\sharp \\
&\leq 0+2 \| \chi _{L\setminus (\bigcap _{t\in T+T}t+L)}\| _{\phi \otimes \mu _L}^\sharp \\
&\leq 2 \| \phi \| ^{1/2} \frac{\mu _G\bigl( L\setminus (\bigcap _{t\in T+T}t+L )\bigr) ^{1/2} }{\mu _G(L)^{1/2}} \\
&< \frac{\delta}{6\mu _G(T)^2}   
\end{align*}
for all $t,s \in T$, $\phi \in \Phi$. The other parts of of the proof are completely same.

(3) In the proof of Theorem 5.5 of \cite{MT}, they show the inequality 
\[ \int _{T_n}^{T_n} \| W^*u_t\alpha _t^n(W)-1\| _2^2 \ dt< 18\epsilon _n .\]
Instead, in the proof of Proposition \ref{2-cohomology}, we show the following inequality.
\begin{align*}
& \int _{T_n}\| W^*u_t\alpha ^n _t(W)-1\| _2^2 \ d\mu _G(t) \\
&\leq \frac{2}{\mu _G(L_n)}\int _{T_n}d\mu _G(t)  \Bigl( \int _{\bigcap _{t\in T_n}t+L_n}\ d\mu _G(s) \  \ \| \tilde{u}_s^*u_t\alpha ^n _t(\tilde{u}_{s-t})-1\| _2^2 \\
&+ \int _{L_n \setminus \bigcap _{t\in T_n}t+L_n}  \ d\mu _G(s) \ \| \tilde{u}_s^*u_t\alpha ^n _t(\tilde{u}_{s-t})-1\| _2^2\Bigr) \\
&\leq \frac{2}{\mu _G(L_n)}\int _{T_n}d\mu _G(t) \int  _{\bigcap _{t\in T_n}t+L_n}d\mu _G(s) \ \| \tilde{u}_s^*u_t\alpha ^n _t(\tilde{u}_{s-t})-1\| _2^2 \\
&+ \frac{8}{\mu _G(L_n)}\mu _G(T_n)\mu _G(L_n \setminus \bigcap _{t\in T_n }t+L_n) \\
& < \frac{2}{\mu _G(L_n)}\int _{T_{n+1}\times T_{n+1}}d\mu _G(t)d\mu _G(s)  \ \| \tilde{u}_s^*u_t\alpha ^n _t(\tilde{u}_{s-t})-1\| _2^2  \\
&+\mu _G(T_n)\frac{{\epsilon _n}^2}{18\mu _G(T_n)^4} \\
&<9\epsilon _n. 
\end{align*}
The other parts of the proof of Proposition \ref{2-cohomology} are same as corresponding parts of the proof of Theorem 5.5 of \cite{MT}. 

(4) In the proof of Proposition \ref{ap 1-coho}, we need to show the inequality 

\[ \| u_t\alpha _t(W)W^* -1\|^{\sharp}_{|\phi |^\omega }\leq 2\| \chi _{L\setminus (\bigcap _{t\in T}t+L )}\|^{\sharp}_{|\phi |\otimes \mu _L}, \]

 which corresponds to the inequality
\[ \| u_t\alpha _t(W)W^* -1\|^{\sharp}_{|\phi |^\omega }\leq 2\frac{t^{1/2}\|\phi \|^{1/2}}{S^{1/2}} \]
in the proof of Theorem 5.11 of \cite{MT}. This is obtained by a similar computation to the above (3).

\bigskip

By using Proposition \ref{2-cohomology}, it is possible to show the following lemma, which corresponds to Lemma 5.8 of \cite{MT}.

\begin{lemm}
\label{x}
Let $\alpha $, $\beta$  be actions with the Rohlin property of a locally compact abelian group $G$ on a factor $M$. Suppose that $\alpha _t\circ \beta _{-t} \in \overline{\mathrm{Int}}(M)$ for all $t\in G$. Let $H$ be a compactly generated clopen subgroup of $G$ and $T$ be a subset of $H$ such that when we identify $H$ with $\mathbf{R}^n \times K\times \mathbf{Z}^m$ for some compact abelian group $K$ and non-negative integers $n$, $m$, $T$ is of the form 
\[ [-t_1,t_1]\times \cdots \times [-t_n,t_n]\times K\times [-M_1,M_1]\times \cdots \times [-M_{m},M_{m}].\]
 Then for any $\epsilon >0$ and a finite set $\Phi \subset M_*$, there exists an $\alpha$-cocycle $u$ such that
\[ \int _T \|\mathrm{Ad}u_t \circ \alpha _t (\phi )-\beta _t(\phi )\| \ d\mu _G(t) <\epsilon \]
for all $\phi \in \Phi$.
\end{lemm}
 In the proof of this lemma, the set corresponding to (5.18) of \cite{MT} is obtained in the following way. For a small positive real number $\eta >0$, take a small number $r>0$ so that 
\[ \| \alpha _t (\phi )-\phi \| <\eta ,\ \| \beta _t(\phi )-\phi \| <\frac{2\eta}{\mu _G(T)}\]
for $\phi \in \Phi$, $t\in G$, $d(t,0)<r$. Choose $A(r,T):=\{ t_j\} _{j=1}^N$ so that for any $t\in T$, there exists $t_j \in A(r,T)$ with $d(t,t_j)<r$. This is possible because $T$ is compact. 

\bigskip

 Now,  we return to the proof of Theorem \ref{main}. The proof is basically the same as that of Case 2 of Lemma 5.12 of \cite{MT}. Here, we only explain the outline. By using Proposition \ref{ap 1-coho} and Lemma \ref{x} alternatively, our main theorem is obtained (the Bratteli--Elliott--Evans--Kishimoto type argument). However, we need to change the following part. In the proof of Case 2 of Lemma 5.12 of \cite{MT}, they take $\{ M_n\}\subset \mathbf{N}$ and $\{ A(M_n, T_n )\}$, which appear in conditions (\textit{n}.1) and (\textit{n}.8). Instead, in (\textit{n}.8), take $r_n\in \mathbf{R}_{>0}$ so that
\[ \| (\hat{v^n}(t)-\hat{v^n}(s))\cdot \phi \| <\epsilon _n, \]
\[ \| \phi \cdot (\hat {v^n}(t)-\hat{v^n}(s) ) \| <\epsilon _n \]
for $t,s \in T_n$, $d(t,s)<r_n$, $\phi \in \hat{\Phi} _{n-1}$. Choose a finite subset $A(r_k, T_k)$ of $T_k$ so that for each $t\in T_k$, there exists $t_0 \in A(r_k, T_k)$ with $d(t,t_0)<r_k$. This is possible because $T_k$ is compact.

\section{Examples}
\label{ex}

Here, we will give  some examples of Rohlin actions. First, we consider actions which fix Cartan subalgebras. This type of examples are classified by Kawahigashi \cite{Kwh2}. One of the most important examples of actions of this form is an infinite tensor product action.

Let $\{ p_n \}$ be a sequence of the dual group $\hat{G}$ of $G$. Set
\[ M:= \bigotimes _{n=1}^{\infty} ( M_2(\mathbf{C}), \mathrm{tr}), \]
 
 Then it is possible to define an action $\theta $ of $G$ by the following way.
 \[ \theta _t :=\bigotimes \mathrm{Ad} \left(
    \begin{array}{ccc}
      1 & 0 \\
      0 & \langle t, p_n \rangle 
    \end{array}
  \right) .\]

Then this $\theta$ has the Rohlin property if and only if the set 
\[ A:=\{ p \in \hat{G} : \mathrm{there\ exists\ a \ subsequence \ of\ } \{ p_n \} \ \mathrm{which\ converges\  to \ } p\} \]
generates a dense subgroup $\Gamma$ in $\hat{G}$. This is seen by the following way. We first show the ``if'' part. Here, we show this implication in the case where for each $p\in A$, a subsequence of  $\{ p_n \}$ which converges to $p$ can be chosen to be a constant sequence.  This case is needed for the proof of Example \ref{cf}. The general case of this implication will follow from Example \ref{cf}. 

Choose $p\in A$. By ignoring other tensor components, we may assume that $p_n=p$ for all $n$. For each $m\in \mathbf{N}$, set
\[ S^m :=\{ \sigma :\{1,\cdots , 2m-1\} \to \{1,2\}\mid \sharp \sigma ^{-1}(1)=m-1, \ \sharp \sigma ^{-1}(2)=m\}.\]
For $\sigma \in S^m$, $m\in \mathbf{N}$ and $k\in \{1, \cdots , 2m-1\}$, set $\tau (k):=3-\sigma (k)$ and
\[ v_\sigma :=e_{\tau(1)\sigma (1)}\otimes \cdots \otimes e_{\tau (2m-1)\sigma (2m-1)}\otimes 1\otimes \cdots .\]
Then we have 
\[ e_\sigma :=v_\sigma^*v_\sigma =e_{\sigma (1) \sigma (1)}\otimes \cdots \otimes e_{\sigma (2m-1)\sigma (2m-1)}\otimes 1 \otimes \cdots,\]
\[ f_\sigma :=v_\sigma v_\sigma ^* =e_{\tau (1) \tau (1)}\otimes \cdots \otimes e_{\tau (2m-1)\tau (2m-1)}\otimes 1 \otimes \cdots,\]
\[ \theta _t (v_\sigma )=\langle t, p \rangle v_\sigma \]
for $t\in G$. Hence if we set
\begin{align*}
 T:=\bigcup _{m=1}^\infty \bigl\{ \sigma \in S^m \mid & \ \sharp ( \sigma ^{-1} (1) \cap \{ 1, \cdots ,k \} ) \\
                                                 & \geq \sharp ( \sigma ^{-1}(2) \cap \{ 1, \cdots ,k \} ) \ \text{for} \ k=1, \cdots , 2m-2 \bigr\} ,
 \end{align*}
then the families $\{ e_\sigma \}_{\sigma \in T}$ and $\{ f_\sigma \}_{\sigma \in T}$ are orthogonal families, respectively. We show that $\sum _{\sigma \in T}e_\sigma =1$, which implies that $\sum _{\sigma \in T}v_\sigma $ is a unitary. This is shown in the following way. Consider the gambler's ruin problem when one has infinite money, the other has no money and they have equal chance to win. Then $\| \sum _{\sigma \in T}e_\sigma \|_1$ is equal to the probability of the poor's ruin. This is $1$. Set
\[ u_n:=1\otimes \cdots \otimes 1\otimes \sum _{\sigma \in T}v_\sigma \in M_2(\mathbf{C})^{\otimes n} \otimes M .\]
Then we have $\{u_n\}\in M_{\omega, \theta}$ and $\theta _t((u_n)_\omega)=\langle t,p\rangle (u_n)_\omega $ for $t\in G$. By assumption, the set $A$ generates a dense subgroup of $\hat{G}$. Hence $\theta $ has the Rohlin property.

 Conversely, assume that the subgroup $\Gamma$ is not dense in $\hat{G}$. Then there exists a non-empty open subset $U$ of $\hat{G}$ with $U\cap \Gamma =\emptyset$. Then by a similar argument to that of the proof of Proposition 1.2 of \cite{Kwh2}, it is shown that the Connes spectrum of $\theta $ and $U$ do not intersect, which implies that $\theta $ does not have the Rohlin property. 
\begin{exam}
\textup{(See also Corollary 1.9 of Kawahigashi \cite{Kwh2})} Let $\alpha$ be an action of a locally compact abelian group $G$ on the AFD factor $R$ of type $\mathrm{II}_1$. Assume that $\alpha $ fixes a Cartan subalgebra of $R$. Then $\alpha $ has the Rohlin property if and only if its Connes spectrum is $\hat{G}$.
\label{cf} 
\end{exam}
 The proof is just a combination of an analogue of Corollary 5.17, which follows from the above example of an infinite tensor product action and Theorem \ref{main}, and Lemma 6.2 of \cite{MT}. In the proof, the crucial fact is that invariantly approximate innerness (see Definition 4.5 of \cite{MT}) is the dual of the Rohlin property. This fact is shown by the completely same argument as in the proof of Theorem 4.11 of \cite{MT}.

 By this example and the main theorem, all the actions fixing Cartan subalgebras with full Connes spectrum are cocycle conjugate to an infinite tensor product action with full Connes spectrum.

The following is a next example.
\begin{exam}
\textup{(See Theorem 6.12 of \cite{MT})} Let $\theta $ be an almost periodic minimal action of a locally compact abelian group $G$ on the AFD factor of type $\mathrm{II}_1$. Then $\theta $ has the Rohlin property.
\end{exam}
\begin{proof}
An almost periodic action is a restriction of a compact abelian group action to its dense subgroup (see Proposition 7.3 of Thomsen \cite{Thom}). If $\theta $ is minimal, then the original compact group action is also minimal, which is unique up to cocycle conjugacy by Jones--Takesaki \cite{JT}. This has the Rohlin property. 
\end{proof}

There exist ``many'' Rohlin actions on AFD factors of type $\mathrm{III}_0$. The following is a generalization of a part of \cite{S2}.
\begin{exam}
\textup{(See Theorem 5 of  \cite{S2})} Let $\alpha $ be an action of a locally compact abelian group $G$ on an AFD factor $M$ and $(\theta ,Z) $ be the flow of weights of $M$. Assume that an action $\{ \mathrm{mod}(\alpha _g) \circ \theta _t\}_{(g,t)\in G\times \mathbf{R}}$ of $G\times \mathbf{R}$ is faithful. Then $\alpha $ has the Rohlin property. Here, $\mathrm{mod}(\alpha )$ denotes the Connes--Takesaki module of $\alpha $ \textup{(}see Connes--Takeskaki \cite{CT}, Haagerup--St\o rmer \cite{HS}\textup{)}.
\end{exam}
\begin{proof}
The proof is the same as that of Case $\mathrm{ker}=0$ of Lemma 6 of \cite{S2}. In the proof of Lemma 6 of \cite{S2}, we use Rohlin's lemma for actions of $\mathbf{R}^2$. Rohlin's lemma for actions of $G\times \mathbf{R}$ also holds, which is shown by the same argument as in Lemma of Lind \cite{L} and Theorem 1 of \cite{FL}.  
\end{proof}

The classification theorem is also applicable to actions of locally compact abelian groups on non-McDuff factors.
\begin{exam}
Let $(X_0, \mu _0)$ be a probability measured space and $\theta _0 \colon G\curvearrowright L^\infty (X_0, \mu _0)$ be a faithful probability measure $\mu _0$-preserving action \textup{(}if $G$ is compactly generated, an example of such an action can be constructed by taking a direct product of increasing compact quotients\textup{)}. Set $D:=\bigotimes _{\mathbf{Z}}L^\infty(X_0, \mu _0)$. Let $\alpha :\mathbf{Z}\curvearrowright D$ be the Bernoulli shift and let $\theta : G\curvearrowright D $ be the diagonal action of $\theta _0$. Then $\theta $ canonically extends to  $M:=D\rtimes _{\alpha *\alpha }\mathbf{F}_2$, which is a non-McDuff factor of type $\mathrm{II}_1$ \textup{(}see, for example, Theorem 10 of Ueda \cite{U2}\textup{)}. Then this action has the Rohlin property.
\end{exam}
\begin{proof}
The proof is the same as that of Theorem 3.3 of \cite{S}.
\end{proof}
Although there are Rohlin actions on non-McDuff factors, as mentioned in \cite{S}, the effect of our classification theorem is limited because $\overline{\mathrm{Int}}(M)$ is too small in many cases.

\bigskip 

\textbf{Acknowledgements.} The author is thankful to Professor Yasuyuki Kawahigashi, who is his adviser, and Professor Reiji Tomatsu for valuable advice which greatly improve the presentation of the paper. He is also thankful to the referee for his careful reading this paper and his useful advice. He is supported by the Program for Leading Graduate Schools, MEXT, Japan.

\end{document}